\documentclass[12pt]{article}
\usepackage[english]{babel}
\usepackage{graphicx}
\usepackage{latexsym}
\usepackage{amsmath, amsthm}
\usepackage{amsfonts}
\usepackage{amssymb}
\usepackage{mathrsfs}
\usepackage{stmaryrd}
\usepackage{psfrag}

\newtheorem{thm}{Theorem}[section]
\newtheorem*{thma}{Theorem A}

\newtheorem*{collar}{Collar Lemma}

\newtheorem{cor}[thm]{Corollary}

\newtheorem{lem}[thm]{Lemma}
\newtheorem{prop}[thm]{Proposition}

\theoremstyle{definition}

\newtheorem{defn}{Definition}
\newtheorem{rem}{Remark}

\newtheorem*{acknowledgement}{Acknowledgement}

\numberwithin{figure}{section}

\newcommand{\PP}{{\mathbb P}}

\newcommand{\C}{{\mathbb C}}

\newcommand{\Z}{{\mathbb Z}}
\newcommand{\D}{{\mathbb D}}

\newcommand{\KK}{{\mathcal K}}

\newcommand{\FF}{{\mathcal F}}

\newcommand{\ra}{{\rightarrow}}

\newcommand{\Int}{{\textup{Int}}}

\newcommand{\Mob}{\textup{M\"ob}}

\begin{document}

\title{Quasiconformal Homogeneity of Genus Zero Surfaces
\thanks{\textup{2000} Mathematics Subject Classification: 30C62 (primary), 30F45 (secondary).}
}

\author{Ferry Kwakkel\thanks{The first author was supported by Marie Curie grant MRTN-CT-2006-035651 (CODY).} 
\and Vladimir Markovic }

\maketitle

\abstract{A Riemann surface $M$ is said to be $K$-quasiconformally homogeneous if for every two points $p,q \in M$, there exists a $K$-quasiconformal
homeomorphism $f \colon M \ra M$ such that $f(p) = q$. In this paper, we show there exists a universal constant $\KK > 1$ such that if $M$ is a 
$K$-quasiconformally homogeneous hyperbolic genus zero surface other than $\D^2$, then $K \geq \KK$. This answers a question by Gehring 
and Palka~\cite{gehring}. Further, we show that a non-maximal hyperbolic surface of genus $g \geq 2$ is not $K$-quasiconformally homogeneous for any
finite $K \geq 1$.} 

\bigskip

\section{Introduction and statement of results}\label{sec_quasi_results}

Let $M$ be a Riemann surface and let $\FF_K(M)$ be the family of all $K$-quasiconformal homeomorphisms of $M$. Then $M$ is said to be $K$-quasiconformally homogeneous
if the family $\FF_K(M)$ is transitive; that is, given any two points $p,q \in M$, there exists an element $f \in \FF_K(M)$ such that $f(p) =q$. In 1976, the notion of quasiconformal homogeneity of Riemann surfaces was introduced by Gehring and Palka~\cite{gehring} in the setting of genus zero surfaces (and higher dimensional analogues). It was shown in~\cite{gehring} that the only genus zero surfaces admitting a transitive conformal family are exactly those which are not hyperbolic, i.e. the surfaces conformally equivalent to either $\PP^1$, $\C$, $\C^* = \C \setminus \{0\}$ or $\D^2$. It was also shown, by means of examples, that there exist hyperbolic genus zero surfaces, homeomorphic to $\PP^1$ minus a Cantor set, that are $K$-quasiconformally homogeneous, for some finite $K >1$. Recently, this problem has received renewed interest. Using Sullivan's Rigidity Theorem, it was shown in~\cite{bonfert_1} by Bonfert-Taylor, Canary, Martin and Taylor, that in dimension $n \geq 3$, there exists a universal constant $\KK_n > 0$ such that for every $K$-quasiconformally homogeneous hyperbolic $n$-manifold other than $\D^n$, it must hold that $K \geq \KK_n$. In~\cite{bonfert_2}, by Bonfert-Taylor, Bridgeman and Canary, a partial result was obtained in dimension two for a certain subclass of closed hyperbolic surfaces satisfying a fixed-point condition. In the setting of genus zero surfaces, notions similar to (but stronger than) quasiconformal homogeneity have been studied by MacManus, N\"akki and Palka in~\cite{macmanus} and further developed in~\cite{bonfert_4} and~\cite{bonfert_3} by Bonfert-Taylor, Canary, Martin, Taylor and Wolf. 

\medskip

In this paper, we answer the original question of Gehring and Palka for genus zero surfaces. 

\begin{thma}[Genus zero surfaces]
There exists a constant $\KK > 1$, such that if $M$ is a $K$-quasiconformally homogeneous hyperbolic genus zero surface other than $\D^2$, 
then $K \geq \KK$.
\end{thma}

The outline of the proof of Theorem A is as follows. First, we restrict our attention to {\em short geodesics}, that is, simple closed geodesics 
which are close in length to the infimum of the lengths of all simple closed geodesics on our surface $M$. For $K>1$ small enough, using $K$-quasiconformal homogeneity, we show there exist intersections of short simple closed geodesics in a small neighbourhood of any preassigned point. Using this information, we construct a configuration of three intersecting short simple closed geodesics, see the three-circle Lemma below. By a combinatorial argument, we show that if $M$ is near conformally homogeneous, these configurations can not exist, leading to the desired contradiction. 

\medskip

As the only genus zero surfaces homogeneous with respect to a conformal family are conformally equivalent to either $\PP^1, \C, \C^*$ or $\D^2$,
we thus have the following corollary of Theorem A.

\begin{cor}
There exists a constant $\KK>1$ such that if $M$ is a $K$-quasiconformally homogeneous genus zero surface with $K < \KK$, 
then $M$ is conformally equivalent to either $\PP^1, \C$, $\C^*$ or $\D^2$.
\end{cor}

This paper is organised as follows. In section 2 we recall some standard material 
and background results that will be used in the remainder of this paper. In section 3 
we present our proof of Theorem A. In section 4 we discuss the problem of obtaining 
a good estimate of the universal constant $\KK > 1$ whose existence we establish in the 
proof of Theorem A. 

\begin{acknowledgement}
The authors thank the referee for several useful suggestions and comments on the manuscript.
\end{acknowledgement}

\section{Preliminaries}\label{sec_quasi_geo_est}

Let us first recall some background material on hyperbolic surfaces and quasiconformal mappings and 
set notation used throughout the remainder of this paper. 

\subsection{Definitions}\label{subsec_defn}

Let $M$ be a hyperbolic Riemann surface. By uniformizing $M$, we may assume that $M = \D^2 \slash \Gamma$ is covered by the Poincar\'e disk $\D^2$. We endow $M$ with the metric $d(\cdot, \cdot)$, induced from the canonical hyperbolic metric $\widetilde{d}(\cdot, \cdot)$ on $\D^2$. The {\em injectivity radius} $\iota(M)$ of $M$ is the infimum over all $p \in M$ of the largest radius for which the exponential map at $p$ is a injective. Define $\lambda(M)$ to be the infimum of the lengths of all simple closed geodesics on $M$. We have that $\lambda(M) \geq 2\iota(M)$. We denote by $D(p,\rho) \subset M$ the closed hyperbolic disk with center $p$ and radius $\rho$. Let $\FF_K(M)$ be the family of all $K$-quasiconformal homeomorphisms of $M$. Then $M$ is said to be $K$-quasiconformally homogeneous if the family $\FF_K(M)$ is transitive; that is, given any two points $p, q \in M$, there exists an element $f \in \FF_K(M)$ such that $f(p) = q$. 

Let us also recall the following. An open Riemann surface $M$ is a said to be {\em extentable} or {\em non-maximal} if it can be embedded in another Riemann surface $M_0$ as a proper subregion; that is, if there exists a conformal mapping of $M$ onto a proper subregion of $M_0$. If $M$ is not extentable, then it is called {\em maximal}. Every open Riemann surface is contained in a maximal Riemann surface (including infinite genus Riemann surfaces),  see~\cite{bochner}. A non-maximal Riemann surface of genus $0$, i.e. embedded in the Riemann sphere, is also called a {\em planar domain}. Proposition~\ref{prop_genus_high} below shows that a non-maximal hyperbolic surface of genus $1 \leq g \leq \infty$ is not 
$K$-quasiconformally homogeneous for any $K \geq 1$. Therefore, as far as quasiconformal homogeneity of hyperbolic surfaces is concerned, 
one can further restrict to either:

\begin{itemize}
\item[\textup{(i)}] hyperbolic genus zero surfaces, or 
\item[\textup{(ii)}] maximal surfaces of genus $2 \leq g \leq \infty$.
\end{itemize}

In this paper, we restrict attention to hyperbolic genus zero surfaces.

\medskip

Given a closed curve $\gamma \in M$, we denote by $[\gamma]$ the homotopy class of $\gamma$ in $M$. The {\em geometric intersection number} 
of isotopy classes of two closed curves $\alpha, \beta \in \pi_1(M)$ is defined by 
\begin{equation}
i(\alpha, \beta) = \min \# \{ \gamma \cap \gamma' \},
\end{equation}
where the minimum is taken over all closed curves $\gamma, \gamma' \subset M$ with $[\gamma] = \alpha$ and $[\gamma'] = \beta$\index{geometric intersection number}. In other words, the geometric intersection number is the least number of intersections between curves representing the two homotopy classes. Let us recall some standard facts about simple closed curves, and in particular simple closed geodesics, on genus zero surfaces, 
see e.g.~\cite{farb}.

\begin{lem}[Curves on genus zero surfaces]\label{lem_curves_properties}
Let $M$ be a genus zero surface.
\begin{enumerate}
\item[\textup{(i)}] Every simple closed curve $\gamma \subset M$ separates $M$ into exactly two connected components.
\item[\textup{(ii)}] If $\gamma, \gamma' \subset M$ are two simple closed curves, then $i(\gamma, \gamma')$ is even.
\item[\textup{(iii)}] If $\gamma, \gamma' \subset M$ are non-homotopic closed geodesics, then the closed curve $\alpha \subset M$ formed by any 
two subarcs $\eta \subset \gamma$ and $\eta' \subset \gamma'$ connecting two points of $\gamma \cap \gamma'$ is homotopically non-trivial. 
\end{enumerate}
\end{lem}

\subsection{Geometrical estimates}

The following lemma describes the asymptotic behaviour of the injectivity radius $\iota(M)$ of a $K$-quasiconformally homogeneous 
hyperbolic surface in terms of $K$, see~\cite{bonfert_1}.

\begin{lem}\label{lem_bounds_inj_rad}
Let $M$ be a $K$-quasiconformally homogeneous hyperbolic surface and $\iota(M)$ its injectivity radius. Then $\iota(M)$ is uniformly bounded from below (for $K$ bounded from above) and $\iota(M) \ra \infty$ for $K \ra 1$. 
\end{lem}

Consequently, $\lambda(M)$ is uniformly bounded from below and $\lambda(M) \ra \infty$ for $K \ra 1$. 
We fix $K_0 > 1$ such that $\lambda(M) \geq 10$ for every $K$-quasiconformally homogeneous hyperbolic surface $M$ with $K \leq K_0$. 

\begin{rem}
If $M$ is a $K$-quasiconformally homogeneous hyperbolic surface, then the fact that $\lambda(M) >0$ implies $M$ has no cusps and therefore, 
any essential simple closed curve $\alpha \subset M$ has a unique simple closed geodesic representative $\gamma \subset M$, 
see e.g.~\cite{thurston}. In particular, if $\gamma \subset M$ is a simple closed geodesic and $f \colon M \ra M$ a homeomorphism, then the 
closed geodesic homotopic to $f(\gamma)$ exists and is simple.
\end{rem}

A {\em pair of pants} is a surface homeomorphic to the sphere $\PP^1$ with the interior of three mutually disjoint closed topological disks removed. Geometrically, it is the surface obtained by gluing two hyperbolic hexagons along their seams. 

In what follows, we denote by $| \gamma |$ the hyperbolic length of a piecewise geodesic curve $\gamma \subset M$. Here by piecewise geodesic curve we mean a finite concatenation of geodesics arcs. We have the following uniform estimate on lengths of simple closed geodesics, see~\cite[Theorem 4.3.3]{fletcher}.

\begin{lem}\label{lem_length_geo}
Let $M$ be a hyperbolic surface and $\gamma$ a simple closed geodesic. Let $f \colon M \ra M$ a $K$-quasiconformal homeomorphism
and $\gamma'$ the simple closed geodesic homotopic to $f(\gamma)$. Then 
\begin{equation}
\frac{1}{K} |\gamma| \leq | \gamma' | \leq K | \gamma |.
\end{equation}
\end{lem}

Further, we use the following classical result in the geometry of hyperbolic surfaces. 

\begin{collar}
Set $m(\ell) = \arcsin(1/(\sinh(\ell/2)))$. For a simple closed geodesic $\gamma \subset M$ of length $\ell = |\gamma|$, the set
\begin{equation}
A(\gamma) = \{p \in M : d(p, \gamma) < m(\ell) \}
\end{equation}
is an embedded annular neighbourhood of $\gamma$.
\end{collar}

In the next two lemma's we recollect uniform approximation estimates of $K$-quasiconformal homeomorphisms, in particular the behaviour 
when $K \ra 1$, see e.g.~\cite{fletcher}.

\begin{lem}\label{lem_approx_homeo_disk}
For every $K \geq 1$ and $0 < \rho < 1$, there exists a constant $C_1(K,\rho)$, depending only on $K$ and $\rho$, such that if 
$f \colon \D^2 \ra \D^2$ is a $K$-quasiconformal homeomorphism and $D(p, \rho) \subset \D^2$ the closed hyperbolic disk of radius $\rho$ 
centered at $p \in \D^2$, there exists a M\"obius transformation $\mu \in \Mob(\D^2)$ such that
\begin{equation}\label{eq_prelim_approx_1}
\widetilde{d}( f(q), \mu(q) ) \leq C_1(K, \rho)
\end{equation}
for all $q \in  D(p, \rho)$. For fixed $\rho >0$, we have that $C_1(K, \rho) \ra 0$ for $K \ra 1$.
\end{lem}

\begin{proof}
By normalizing with suitable M\"obius transformations, we may assume that $f(0) = 0$ and $f(1) = 1$.  As the family of $K$-quasiconformal homeomorphisms of $\D^2$ onto itself fixing $0$ and $1$ in $\D^2$ is a normal family, see e.g.~\cite[p. 32]{ahlfors}, by a standard argument of uniform convergence on compact subsets, there exists a function $C_1(K,\rho)$ with $C_1(K, \rho) \ra 0$ if $K \ra 1$ such that $\widetilde{d}(f(z), z) \leq C_1(K, \rho)$.  Thus~\eqref{eq_prelim_approx_1} follows.
\end{proof}

Further, we will utilize the following, see e.g.~\cite[Lemma 2]{buser}.

\begin{lem}\label{lem_approx_geodesic}
For every $K \geq 1$, there exists a constant $C_2(K)$ depending only on $K$ with the following property. 
Let $M$ be a hyperbolic surface and $\gamma \subset M$ a simple closed geodesic.  
If $f \colon M \ra M$ is a $K$-quasiconformal homeomorphism, then the geodesic $\gamma'$ homotopic to 
$f(\gamma)$ has the property that 
\begin{equation}
d( f(\gamma), \gamma' ) \leq C_2(K).
\end{equation}
Furthermore, $C_2(K) \ra 0$ as $K \ra 1$.
\end{lem}

\subsection{Non-maximal surfaces of positive genus}

Using the geometrical estimates derived in the previous section, let us first prove the claim made in the introduction that non-maximal Riemann surfaces of 
genus $1 \leq g \leq \infty$ are not $K$-quasiconformally homogeneous for any finite $K \geq 1$.

\begin{prop}[Non-maximal surfaces of positive genus]\label{prop_genus_high}
Let $M$ be a non-maximal surface of genus $1 \leq g \leq \infty$. Then $M$ is not $K$-quasiconformally homogeneous for any $K \geq 1$.
\end{prop}

\begin{proof}
To derive a contradiction, assume that $M$ is $K$-quasiconformally homogeneous for some finite $K \geq 1$. As $M$ is non-maximal, $M$ is 
embedded in a maximal hyperbolic surface $M_0$ of genus $g \geq 1$. Let $\bar{p} \in M_0$ be an ideal boundary point of $M$ and let $D \subset M_0$ 
be a small closed disk embedded in $M_0$ and centered at $\bar{p}$. There exists a sequence of points $p_n \in M \cap D$ 
so that $d_0(p_n, \bar{p}) \ra 0$ (where $d_0$ is the metric on $M_0$) and thus $d(p_n, \partial D) \ra \infty$, as $\bar{p}$ is in the ideal 
boundary of $M$. On the other hand, as $g \geq 1$, there exists a non-separating simple closed geodesic $\gamma \subset M$. Mark a point 
$p \in \gamma$. By transitivity of the family $\FF_K(M)$, for every $n \geq 1$, there exists an element $f_n \in \FF_K(M)$ such that $f_n(p) = p_n$. Therefore the simple closed curve $f_n(\gamma)$ is non-separating for every $n \geq 1$ and thus
\begin{equation}\label{lem_length_gamma_*}
f_n(\gamma) \cap \partial D \neq \emptyset.
\end{equation}
Indeed, otherwise we have that $f_n(\gamma) \subset M \cap D$, implying that $f_n(\gamma)$ is separating, as $D$ is an embedded disk in the maximal surface $M_0$ and thus the connected components of $M \cap D$ are planar subsurfaces, contradicting that $\gamma$ (and thus $f_n(\gamma)$) is non-separating. 
By Lemma~\ref{lem_approx_geodesic}, the geodesic $\gamma_n$ homotopic to $f^n (\gamma)$ has to stay within a bounded distance of $f_n(\gamma)$ 
and therefore, for $n$ large enough, the geodesic $\gamma_n$ has the property that 
\begin{equation}\label{lem_length_gamma_**}
\gamma_n \cap \partial D \neq \emptyset
\end{equation}
by~\eqref{lem_length_gamma_*}. It follows from~\eqref{lem_length_gamma_**} that
\begin{equation}
|\gamma_n| \geq 2 ( d(p_n, \partial D) - C_2(K) ) 
\end{equation}
with $C_2(K)$ the uniform constant of Lemma~\ref{lem_approx_geodesic}; put in words, the geodesic $\gamma_n$ has to enter $D \cap M$, 
pass close to $p_n$, and leave $D \cap M$ again. As $d(p_n, \partial D) \ra \infty$ for $n \ra \infty$, it follows that $| \gamma_n | \ra \infty$ for $n \ra \infty$, contradicting Lemma~\ref{lem_length_geo}. Thus $M$ can not be $K$-quasiconformally homogeneous for any finite $K \geq 1$. 
\end{proof}

\section{Quasiconformal homogeneity of genus zero surfaces}

In the remainder of our proof, let $K_0>1$ as defined in section~\ref{sec_quasi_geo_est} and let $M$ be a $K$-quasiconformally homogeneous hyperbolic genus zero surface, 
with $1 < K \leq K_0$, and $\FF_K(M)$ the family of all $K$-quasiconformal homeomorphisms of $M$, which is transitive by homogeneity of $M$.

\subsection{The two-circle Lemma}\label{subsec_two_circle}

In what follows, we focus on {\em short geodesics}\index{short geodesics}, in the following sense.

\begin{defn}[$\delta$-short geodesics]
Given $\delta > 0$, a simple closed geodesic $\gamma \subset M$ is said to be {\em $\delta$-short} if $| \gamma | \leq (1+ \delta) \lambda(M)$.
\end{defn}

By Lemma~\ref{lem_bounds_inj_rad}, and the remark following it, for every $K \leq K_0$, there is a uniform lower bound on the length 
of simple closed geodesics on $M$. Fix 
\begin{equation}\label{eq_fix_delta_0}
\delta_0 = \frac{1}{378}.
\end{equation}

\begin{defn}[Two-circle configuration]\index{two-circle configuration}
A {\em two-circle configuration} is a union of two $\delta_0$-short geodesics $\gamma_1, \gamma_2 \in M$ such that $\gamma_1$ and $\gamma_2$ 
intersect in exactly two points $p_1, p_2 \in M$. 
\end{defn}

Topologically a two-circle configuration is a union of two simple closed curves $\gamma_1, \gamma_2$ in the surface $M$, intersecting transversely in exactly two points and 
\[ M \setminus (\gamma_1 \cup \gamma_2) \] 
consists of four connected components and the boundary of each component consists of two arcs. For future reference, let us label the four arcs $\eta_1, \eta_2 \subset \gamma_1$ and $\eta_3, \eta_4 \subset \gamma_2$ connecting the two intersection points $p_1$ and $p_2$, see Figure~\ref{fig_two_circle}.

\begin{figure}[h]
\begin{center}
\psfrag{gamma_1}{$\gamma_1$}
\psfrag{gamma_2}{$\gamma_2$}
\psfrag{eta_1}{$\eta_1$}
\psfrag{eta_1'}{$\eta_2$}
\psfrag{eta_2}{$\eta_3$}
\psfrag{eta_2'}{$\eta_4$}
\psfrag{p_1}{$p_1$}
\psfrag{p_2}{$p_2$}
\includegraphics[scale=0.7]{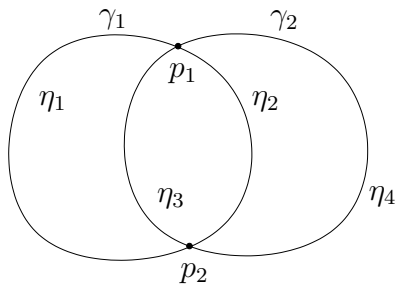}
\caption{A two-circle configuration with labeling.}
\label{fig_two_circle}
\end{center}
\end{figure}

First, the following result, see also~\cite[Proposition 4.6]{parlier}.

\begin{lem}\label{lem_pants_disk}
There exists a uniform constant $r_0 > 0$ such that for a pair of pants $P \subset M$, there exists a $p \in P$ such that
$D(p, r_0) \subset P$. 
\end{lem}

\begin{proof}
Each pair of pants decomposes into two ideal triangles. As every ideal triangle contains a disk of radius $\frac{1}{2} \log 3$, every pair of pants therefore contains a disk 
of (at least) that radius. Thus we may take $r_0 = \frac{1}{2} \log 3$.
\end{proof}

We first prove the existence of intersecting $\delta_0$-short geodesics on $M$ for sufficiently small $K>1$, as we build forth upon this result 
in the remainder of the proof.

\begin{lem}[Intersections of short geodesics]\label{lem_intersections}
There exists a constant $1 < K_1 \leq K_0$, such that if $M$ is $K$-quasiconformally homogeneous with $1 < K \leq K_1$, then there exist 
$\delta_0$-short geodesics that intersect.
\end{lem}

\begin{proof}
To prove there exist intersecting $\delta_0$-short geodesics on $M$, we argue by contradiction. That is, suppose all $\delta_0$-short geodesics on $M$ are mutually disjoint. Choose $1 < K_1 \leq K_0$ so that 
\begin{equation}\label{eq_choice_K}
K_1 \leq \frac{1 + \delta_0}{1+ \delta_0 / 2} ~ \textup{and} ~ C_2(K_1) \leq \frac{r_0}{2},
\end{equation}
with $C_2(K)$ the constant of Lemma~\ref{lem_approx_geodesic} and $r_0$ the constant of Lemma~\ref{lem_pants_disk}. 

Let us first observe that there exist infinitely many distinct $\delta_0$-short geodesics on $M$. Indeed, as $\lambda(M)>0$, there exists a simple closed geodesic $\gamma_0 \subset M$ such that $|\gamma_0| \leq (1+ \delta_0/2) \lambda(M)$. Mark a point $p_0 \in \gamma_0$ and choose $f \in \FF_K(M)$ such that $f(p_0)=q$, for a certain $q \in M$. By our choice of $K_1$, cf.~\eqref{eq_choice_K}, combined with Lemma~\ref{lem_length_geo}, the geodesic $\gamma$ homotopic to $f(\gamma_0)$ is $\delta_0$-short, for every $f \in \FF_K(M)$. As the surface is unbounded (in the hyperbolic metric), by transporting the geodesic $\gamma_0$ by different elements of $\FF_K(M)$ sufficiently far apart, by Lemma~\ref{lem_approx_geodesic}, we see that there must indeed exist infinitely many different $\delta_0$-short curves. Denote $\Gamma_0$ the (countable) family of all $\delta_0$-short geodesics on $M$. 

As all elements of $\Gamma_0$ lie in different homotopy classes, and all elements are mutually disjoint, we claim that the elements of $\Gamma_0$ 
are locally finite, in the sense that a compact subset of $M$ only intersects finitely many distinct elements of $\Gamma_0$. Indeed, 
suppose that a compact subset of $M$ intersects infinitely many elements of $\Gamma_0$. Label these geodesics $\gamma_n$, $n \in \Z$. 
By compactness, there exists an element $\gamma :=\gamma_n$, for some $n \in \Z$, and a subsequence $\gamma_{n_k}$, with $n_k \neq n$, such that $d(\gamma, \gamma_{n_k}) \ra 0$ for $k \ra \infty$. As all these elements are mutually disjoint, we can find points $p_k \in \gamma_{n_k}$ such that $p_k \ra p \in \gamma$, where, moreover, the vectors $v_k \in T_{p_k}M$ tangent to $\gamma_{n_k}$ at $p_k$ converge to the tangent vector $v \in T_{p}M$ of $\gamma$ at $p$. As the lengths of the geodesics $\gamma_{n_k}$ are uniformly bounded from above, by the Collar Lemma (see section~\ref{sec_quasi_geo_est}), every curve $\gamma_{n_k}$ is contained in a uniformly thick embedded annulus $A_k := A(\gamma_{n_k}) \subset M$. 
Conversely, as the lengths of the geodesics are uniformly bounded from above, and as the initial data $(p_k, v_k)$ of $\gamma_{n_k}$ converges to 
the initial data $(p,v)$ of $\gamma$, the geodesics $\gamma_{n_k}$ converge uniformly to $\gamma$. In particular, for sufficiently large $k$, $\gamma$ is entirely contained in $A_k$. However, this implies that $\gamma$ is homotopic to $\gamma_{n_k}$, a contradiction as these were all assumed to be mutually disjoint and thus non-homotopic.

Choose an element $\gamma_1 \in \Gamma_0$. As the elements of $\Gamma_0$ are locally finite, and as the distance between any two elements of 
$\Gamma_0$ is finite, the distance between $\gamma_1$ and the union of the elements $\Gamma_0 \setminus \{ \gamma_1 \}$ is therefore bounded from below and above. In particular, there exists a $\delta_0$-short geodesic with shortest distance to $\gamma_1$ (though this geodesic need not be unique). Denote one such geodesic by $\gamma_2$. There exists a geodesic arc $\eta \subset M$ connecting $\gamma_1$ and $\gamma_2$, with $|\eta| = d(\gamma_1, \gamma_2)$. Take a simple closed curve $\alpha \subset M$ homotopic to $\gamma_1 \cup \eta \cup \gamma_2$ and let $\gamma'$ be the (not necessarily $\delta_0$-short) geodesic homotopic to $\alpha$. As $\gamma_1$ and $\gamma_2$ are disjoint, and therefore in distinct homotopy classes, $\gamma'$ is non-trivial. By Lemma~\ref{lem_curves_properties} (iii), we have that 
\[ \gamma' \cap ( \gamma_1 \cup \gamma_2) = \emptyset. \]
Let $P \subset M$ be the pair of pants bounded by the simple closed geodesics $\gamma', \gamma_1$ and $\gamma_2$. 
As every pair of pants contains a unique simple geodesic arc connecting each pair of boundary geodesics of the pair of pants, $\eta \subset P$ is the unique geodesic arc in $P$ joining $\gamma_1$ and $\gamma_2$ such that $|\eta| = d(\gamma_1, \gamma_2)$, see Figure~\ref{fig_pants_homotopic}.

\begin{figure}[h]
\begin{center}
\psfrag{gamma_1}{$\gamma_1$}
\psfrag{gamma_2}{$\gamma_2$}
\psfrag{eta}{$\eta$}
\psfrag{gamma_3}{$\gamma'$}
\psfrag{alpha}{$\alpha$}
\includegraphics[scale=0.6]{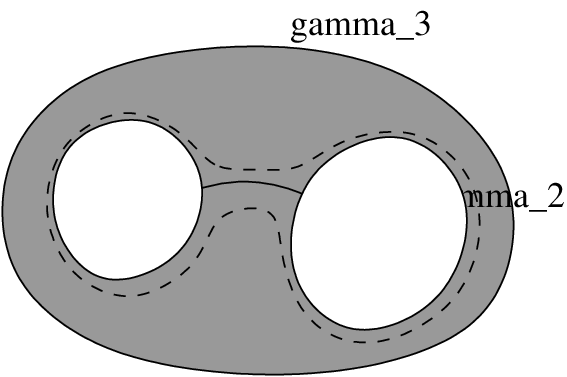}
\caption{Proof of Lemma~\ref{lem_intersections}}
\label{fig_pants_homotopic}
\end{center}
\end{figure}

Next, we claim that the interior of $P$ is disjoint from any $\delta_0$-short geodesic. Indeed, let 
\[ \gamma_3 \in \Gamma_0 \setminus \{ \gamma_1,\gamma_2\}, \]
and suppose that $\gamma_3 \cap \Int(P) \neq \emptyset$. As $\gamma_1, \gamma_2 \in \Gamma_0$, by assumption $\gamma_3$ can not intersect $\gamma_1$ or $\gamma_2$. Thus, if $\gamma_3 \cap \Int(P) \neq \emptyset$, then we necessarily have that $\gamma_3 \cap \gamma' \neq \emptyset$. 
Consider any two consecutive intersection points $q_1, q_2 \in \gamma_3$ of $\gamma'$ with $\gamma_3$ and denote $\eta' \subset \gamma_3  \cap P$ the corresponding 
simple arc with endpoints $q_1$ and $q_2$. We first show that we must have that $\eta' \cap \eta \neq \emptyset$. To show this, suppose that 
$\eta' \cap \eta = \emptyset$. Then $\eta' \cap (\gamma_1 \cup \eta \cup \gamma_2) = \emptyset$. Define $\gamma' _{1,2} \subset \gamma' $ to be the simple arcs which 
are the two connected components of $\gamma'  \setminus \{ q_1, q_2\}$. Consider the subsurfaces $M_1, M_2 \subset M$ bounded by $\gamma'_1 \cup \eta'$ 
and $\gamma'_2 \cup \eta'$ respectively and intersecting $\Int(P)$. As $\eta' \cap (\gamma_1 \cup \eta \cup \gamma_2) = \emptyset$, $\gamma_1 \cup \eta \cup \gamma_2$ is contained in either $M_1$ or $M_2$. However, this implies that one of these subsurfaces is a topological disk, which contradicts Lemma~\ref{lem_curves_properties} (iii). Therefore, we must have that $\eta' \cap \eta \neq \emptyset$. This in turn implies that 
\begin{equation}
d(\gamma_3, \gamma_1) < d(\gamma_1, \gamma_2 )
\end{equation}
which contradicts the assumption that $\gamma_2$ is the closest $\delta_0$-short geodesic to $\gamma_1$. Thus the interior of $P$ is disjoint 
from any $\delta_0$-short geodesic.

By Lemma~\ref{lem_pants_disk}, there exists a point $p \in P$ such that $D(p, r_0) \subset P$. By the previous paragraph, the disk $D(p,r_0)$ is disjoint from any $\delta_0$-short geodesic. Take $f \in \FF_K(M)$ such that $f(p_0)= p$. The geodesic $\gamma''$ homotopic to $f(\gamma_0)$ is 
$\delta_0$-short and, again by our choice of $K_1$, combined with Lemma~\ref{lem_approx_geodesic}, the geodesic $\gamma''$ has the property that 
$\gamma'' \cap D(p, r_0) \neq \emptyset$. This contradicts our earlier conclusion that the interior of $P$ is disjoint from $\delta_0$-short geodesics and thus there must exist $\delta_0$-short geodesics that intersect.
\end{proof}

\begin{lem}[Two-circle Lemma]\label{lem_two_circles}
Let $M$ be $K$-quasiconformally homogeneous with $1 < K \leq K_1$ and let $\gamma_1$ and $\gamma_2$ be two intersecting $\delta$-short geodesics,
where $\delta < 1/3$. Then $\gamma_1 \cup \gamma_2$ is a two-circle configuration and the four arcs $\eta_i$, $1 \leq i \leq 4$, connecting the intersection points $p_1$ and $p_2$ have lengths
\begin{equation}\label{eq_two_circle}
\frac{\lambda(M)}{2} - \frac{\delta}{2} \lambda(M) \leq | \eta_i | \leq \frac{\lambda(M)}{2} + \frac{3\delta}{2} \lambda(M).
\end{equation}
\end{lem}

\begin{proof}
Let $\gamma_1, \gamma_2 \subset M$ be two $\delta$-short geodesics that intersect. By Lemma~\ref{lem_curves_properties} (ii), $\gamma_1$ and 
$\gamma_2$ intersect in an even number of points. To prove there can be no more than two intersection points, suppose that there are $2k$
intersection points, with $k \geq 2$. Label these points $p_1,...,p_{2k}$ according to their cyclic ordering on $\gamma_1$, relative to an orientation on $\gamma_1$ and an initial point. Define the arcs $\alpha_i$, with $1 \leq i \leq 2k$, to be the connected components of 
\[ \gamma_1 \setminus \bigcup_{i=1}^{2k} p_i. \] 
As $k \geq 2$ by assumption, at least one of these arcs has length at most $(1+ \delta)\lambda(M)/4$. Without loss of generality, we may suppose that this is the case for $\alpha_1$. Then the endpoints of $\alpha_1$, $p_1$ and $p_2$, cut the geodesic $\gamma_2$ into two connected components $\beta_1$ and $\beta_2$. One of these components, say $\beta_1$, has length at most $(1+ \delta)\lambda(M)/2$. 
By Lemma~\ref{lem_curves_properties} (iii), $\alpha_1 \cup \beta_1$ is a non-trivial closed curve. However, we have that
\[ |\alpha_1 \cup \beta_1| \leq \frac{3(1+ \delta)\lambda(M)}{4} < \lambda(M), \]
as $(1+\delta) < 4/3$, which is impossible. Thus $\gamma_1$ and $\gamma_2$ intersect in exactly two points.

To prove~\eqref{eq_two_circle}, we adopt the labeling in Figure~\ref{fig_two_circle}. As $\gamma_1 = \eta_1 \cup \eta_2$ and 
$| \gamma_1| \leq (1+ \delta)\lambda(M)$, one of the arcs $\eta_1$ or $\eta_2$ has length at most $(1+ \delta)\lambda(M)/2$. We may assume this is the case for $\eta_1$. As the closed curve $\eta_3 \cup \eta_1$ is homotopically non-trivial, we must have that
\[ \frac{(1+\delta)\lambda(M)}{2} + | \eta_3 | \geq |\eta_1| + |\eta_3| \geq \lambda(M), \]
and thus
\begin{equation}\label{eq_two_est_1}
|\eta_3| \geq \frac{(1 - \delta)\lambda(M)}{2}.
\end{equation}
Conversely, in order that $| \eta_3 | + |\eta_4| \leq (1+\delta)\lambda(M)$, by~\eqref{eq_two_est_1}, we must
have that 
\begin{equation}
| \eta_4 | \leq \left( \frac{1}{2} + \frac{3}{2} \delta \right) \lambda(M).
\end{equation}
The other cases follow by symmetry. This finishes the proof.
\end{proof}

In particular, the two-circle Lemma holds for all $\delta \leq 6 \delta_0 < 4/3$. For future reference, we introduce the following. 

\begin{defn}[Tight pair of pants]\index{pair of pants!tight pair of pants}
A {\em tight pair of pants} is a pair of pants $P \subset M$ such that the three boundary curves are $3 \delta_0$-short geodesics.
\end{defn}

We have the following corollary of the two-circle Lemma.

\begin{cor}\label{cor_exist_pants}
Let $M$ be $K$-quasiconformally homogeneous with $1 < K \leq K_1$. Then there exists a tight pair of pants.
\end{cor}

\begin{figure}[h]
\begin{center}
\psfrag{gamma_1}{$\gamma_1$}
\psfrag{gamma_2}{$\gamma_2$}
\psfrag{alpha_1}{$\alpha_1$}
\psfrag{alpha_2}{$\alpha_2$}
\psfrag{gamma'_1}{$\gamma_1'$}
\psfrag{gamma'_2}{$\gamma_2'$}
\psfrag{P}{$P$}
\psfrag{p_1}{$p_1$}
\psfrag{p_2}{$p_2$}
\includegraphics[scale=0.7]{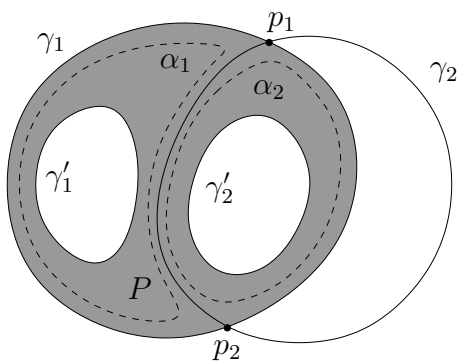}
\caption{Proof of Corollary~\ref{cor_exist_pants}.}
\label{fig_pair_pants}
\end{center}
\end{figure}

\begin{proof}
By the two-circle Lemma, for $K \leq K_1$, there exists $\delta_0$-short geodesics $\gamma_1$ and $\gamma_2$ 
that intersect in exactly two points. In the labeling of Figure~\ref{fig_two_circle}, let $\alpha_1$ be the simple closed curve $\eta_1 \cup \eta_3$
and $\gamma'_1$ be the simple closed geodesic homotopic to $\alpha_1$. Similarly, let $\alpha_2$ the simple closed curve $\eta_2 \cup \eta_3$ and $\gamma_2'$ be the simple closed geodesic homotopic to $\alpha_2$. By Lemma~\ref{lem_curves_properties} (iii), the three geodesics $\gamma'_1, 
\gamma'_2$ and $\gamma_1$ are disjoint and thus the region bounded by these three simple closed geodesics is a pair of pants $P$, 
see Figure~\ref{fig_pair_pants}. 

To prove $P$ is a tight pair of pants, it suffices to show that $\gamma'_i$ is $3\delta_0$-short, for $i=1,2$, as $\gamma_1$ is $\delta_0$-short. 
It follows from~\eqref{eq_two_circle} of the two-circle Lemma that 
\begin{equation}
| \eta_1 | + |\eta_3| \leq 2 \left( \frac{\lambda(M)}{2} + \frac{3\delta_0}{2}\lambda(M) \right) = (1+ 3\delta_0)\lambda(M).
\end{equation}
Therefore, the length of $\gamma'_1$ is bounded by the length of $\eta_1 \cup \eta_3$, which is at most $(1+ 3\delta_0)\lambda(M)$. 
Similarly, considering the length of $\eta_2 \cup \eta_4$, we obtain that $\gamma_2'$ is $3 \delta_0$-short. Therefore, $P$ is a tight pair 
of pants.
\end{proof}

\subsection{Definite angles of intersection}\label{subsec_quasi_def_angle}

In what follows, we use the following notation. Let $\gamma_1, \gamma_2 \subset M$ be two simple closed geodesics that intersect at a point $p \in M$. The angle between the two geodesics at $p \in M$, denoted $\angle ( \gamma_1, \gamma_2 )_p$, is defined to be the minimum of $\angle ( v_1, v_2 )_p$ and $\angle (v_1, -v_2)_p$ where $v_1,v_2 \in T_pM$ is a tangent vector to $\gamma_1, \gamma_2$ respectively. In order to produce certain configurations of intersecting simple closed geodesics, we show that for all $K > 1$ sufficiently small, there exist $3 \delta_0$-short geodesics intersecting at a uniformly large angle. More precisely, 

\begin{lem}[Definite angles of intersection]\label{lem_large_angle}
There exists a constant $1 < K_2 \leq K_1$, such that if $M$ is $K$-quasiconformally homogeneous with $1 < K \leq K_2$, 
then there exist two $3 \delta_0$-short geodesics $\gamma_1, \gamma_2 \subset M$, intersecting at a point $q \in M$, such that 
$\angle ( \gamma_1, \gamma_2 )_q \geq \pi/4$.
\end{lem}

The proof of Lemma~\ref{lem_large_angle} uses the following two auxiliary lemma's.

\begin{lem}\label{lem_hexagon}
Let $H \subset \D^2$ be a right-angled hyperbolic hexagon. Let $a,b,c$ be the sides of the alternate edges and $a', b', c'$ the sides of the opposite edges. Suppose that $|a| = (1+ \epsilon_1) \ell, |b| = (1+ \epsilon_2) \ell$ and $|c| = (1+\epsilon_3) \ell$ with $0 \leq \epsilon_i < 1/2$, 
$1 \leq i \leq 3$. For every $\epsilon >0$, there exists $\ell_{\epsilon} >0$ such that, if $\ell \geq \ell_{\epsilon}$, then the lengths of the sides $a',b'$ and $c'$ are at most $\epsilon$.
\end{lem}

\begin{proof}
By the hyperbolic cosine law for right-angled hexagons (see~\cite{ratcliffe}), we have that 
\begin{equation}
\cosh (|c'|) = \frac{\cosh (|a|) \cosh (|b|) + \cosh (|c|)}{\sinh(|a|) \sinh(|b|)},
\end{equation}
which we can write as 
\begin{equation}
\cosh (|c'|) = \frac{1}{\tanh(|a|) \tanh(|b|)} + \frac{\cosh (|c|)}{\sinh(|a|) \sinh(|b|)}
\end{equation}
We have that 
\begin{equation}
\frac{\cosh (|c|)}{\sinh(|a|) \sinh(|b|)} \asymp \frac{ e^{(1+\epsilon_3)\ell} }{ e^{(2+\epsilon_1+\epsilon_2 )\ell} } \ra 0, 
\quad \textup{for} \quad \ell \ra \infty,
\end{equation}
as $0 \leq \epsilon_i < 1/2$ for $1 \leq i \leq 3$. Further, as $\tanh(r) \ra 1$ for $r \ra \infty$, given any $\epsilon' >0$, there exists an $\ell_{\epsilon'}$ such that, if $\ell \geq \ell_{\epsilon'}$, then
\begin{equation}\label{eq_epsilon'}
\cosh (|c'|) \leq 1 + \epsilon'.
\end{equation}
Thus given an $\epsilon>0$, choose $\ell_{\epsilon'}$ such that~\eqref{eq_epsilon'} is satisfied with $\epsilon := \cosh^{-1}(1+ \epsilon')$.
For this $\ell_{\epsilon'}$ (with $\epsilon'$ depending on $\epsilon$ only), we have that
\begin{equation}
|c'| \leq \cosh^{-1}(1+ \epsilon') = \epsilon 
\end{equation}
Cyclically permuting $a, b,c$ and $a', b', c'$ gives a similar estimate for $a'$ and $b'$ and this finishes the proof.
\end{proof}

\begin{lem}\label{lem_ideal_triangle}
Let $T \subset \D^2$ be an ideal triangle with boundary $\partial T = \gamma_1 \cup \gamma_2 \cup \gamma_3$ and barycenter $0 \in \D^2$.
Let $\gamma \subset \D^2$ be a geodesic passing through $0 \in \D^2$. Then for an $i \in \{1,2,3\}$, $\gamma$ intersects $\partial T$ at a 
point $p \in \gamma_i$, such that $\pi/4 + \epsilon_0 \leq \angle( \gamma, \gamma_i )_p \leq \pi/2 $, where $\epsilon_0 \approx 0.24$.
\end{lem}

\begin{figure}[h]
\begin{center}
\psfrag{gamma_1}{$\gamma_1$}
\psfrag{gamma_2}{$\gamma_2$}
\psfrag{gamma_3}{$\gamma_3$}
\psfrag{gamma}{$\gamma$}
\psfrag{theta}{$\theta_1$}
\psfrag{ell}{$l_1$}
\psfrag{ell_0}{$l_0$}
\psfrag{D^2}{$\D^2$}
\psfrag{H}{$T$}
\psfrag{p_c}{$0$}
\psfrag{p}{$p$}
\includegraphics[scale=0.7]{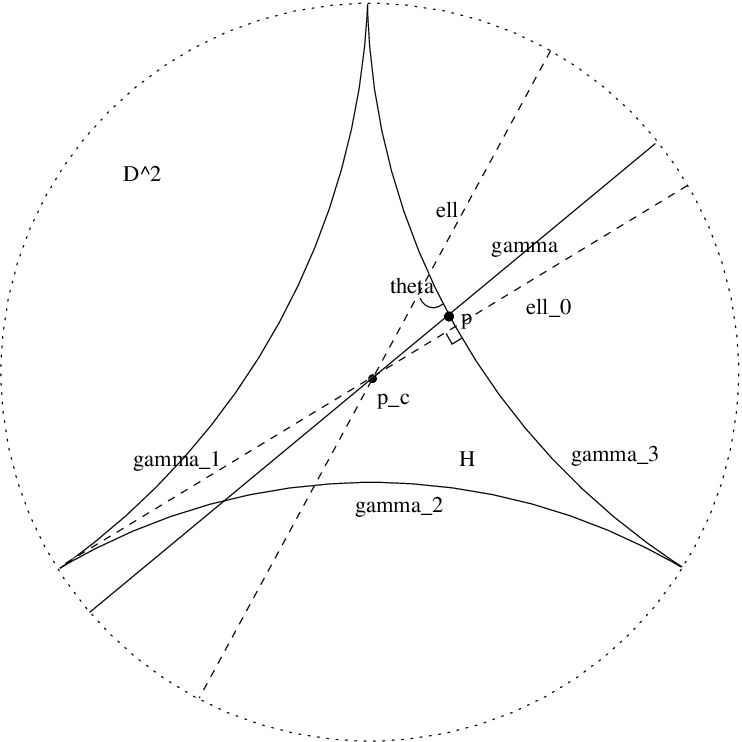}
\caption{Proof of Lemma~\ref{lem_ideal_triangle}.}
\label{fig_ideal_triangle}
\end{center}
\end{figure}

\begin{proof}
Let us label the boundary geodesics $\gamma_1, \gamma_2, \gamma_3$ of the ideal triangle $T$ as in Figure~\ref{fig_ideal_triangle}. 
Let $l_0 \subset \D^2$ be the axis of symmetry of $T$, relative to the symmetry that exchanges $\gamma_1$ and $\gamma_2$.
The distance of $0 \in \D^2$ to the point of intersection of $\ell_0$ with $\gamma_3$ is $\frac{1}{2} \log 3$. Let $l_1$ be the axis perpendicular
to the axis of symmetry relative to the symmetry that exchanges $\gamma_2$ and $\gamma_3$, see Figure~\ref{fig_ideal_triangle}.
Let $\theta_1$ be the angle between the axis $l_1$ and the geodesic $\gamma_3$. The angle between the axes $l_0$ and $l_1$ is 
$\pi/6$. Therefore, by the hyperbolic cosine law, the angle $\theta_1$ between $\gamma_3$ and $l_1$ is given by 
\[ \cos ( \theta_1 ) = \cosh \left( \frac{1}{2} \log 3 \right) \sin \left( \frac{\pi}{6} \right). \]
It is readily verified that $\theta_1 = \pi/4 + \epsilon_0$, with $\epsilon_0 \approx 0.24$. Suppose that the geodesic $\gamma$ that passes through $0 \in \D^2$ is such that $\gamma \cap \gamma_3 \neq \emptyset$. Denote $p$ the point of intersection of $\gamma$ with $\gamma_3$. By symmetry of the configuration, we may assume that $p$ is contained in the arc $\eta \subset \gamma_3$ cut out by the two intersection points of $l_0$ and $l_1$ with $\gamma_3$, which, up to symmetry, represents the extremal 
case. Thus we have that $\pi/4 + \epsilon_0 \leq \angle( \gamma, \gamma_i )_p \leq \pi/2$. 
This finishes the proof of the Lemma.
\end{proof}

\begin{proof}[Proof of Lemma~\ref{lem_large_angle}]
By Corollary~\ref{cor_exist_pants}, there exists a tight pair of pants $P \subset M$, i.e. a pair of pants $P$ with $\partial P = \gamma_1 \cup \gamma_2 \cup \gamma_3$, with the property that the three boundary geodesics $\gamma_1 , \gamma_2, \gamma_3$ are $3 \delta_0$-short. 
Lifting the pair of pants $P$ to the cover $\D^2$, it unfolds to two right-angled hexagons $H, H' \subset \D^2$, each of which contains exactly a half of the component of the lift $\widetilde{\gamma}_i$ of $\gamma_i$ to $\D^2$ as its alternate boundary arcs, with $1 \leq i \leq 3$. Restrict 
to $H$ and denote $\eta_i$ with $1 \leq i \leq 3$ the alternating boundary arcs. As the length of $\eta_i$ is exactly half of that of $\gamma_i$, 
with $1 \leq i \leq 3$, and these are $3 \delta_0$-short, the lengths $\eta_i$ satisfy the length requirement of Lemma~\ref{lem_hexagon},
so that, given any $\epsilon > 0$ there exists a $K>1$ such that the lengths of the three sides of the hexagon $H$ opposite to $\eta_i$ are of 
length at most $\epsilon$, as $\lambda(M) \ra \infty$ for $K \ra 1$ by Lemma~\ref{lem_bounds_inj_rad}.

Therefore, by normalizing by a suitable element of $\Mob(\D^2)$ if necessary, the hexagon $H \subset \D^2$ converges on a compact disk $D(0, 10)$ to an ideal triangle with barycenter $0 \in \D^2$, for $K \ra 1$. In particular, by Lemma~\ref{lem_ideal_triangle}, there exists 
$1 < K_2 \leq K_1$, such that any geodesic $\gamma' \subset \D^2$ passing through $0 \in \D^2$ intersects one of the arcs $\eta_i$ of $H$ 
under an angle at least $\pi/4 + \epsilon_0/2$, with $\epsilon_0$ as in Lemma~\ref{lem_ideal_triangle}. As geodesics in $\D^2$ passing near $0 \in \D^2$ are almost straight lines, there exists an $\epsilon_1 >0$ such that any geodesic $\gamma' \subset \D^2$ passing through the disk $D(0,\epsilon_1) \subset \D^2$ intersect one of the arcs $\eta_i$ at an angle at least $\pi/4$. By choosing $K_2>1$ smaller if necessary, we can be sure that $C_2(K_2) \leq \epsilon_1$, where $C_2(K)$ is the constant of Lemma~\ref{lem_approx_geodesic}.

Let $\gamma_0 \subset M$ be a $\delta_0$-short geodesic and let $p_0 \in \gamma_0$. Choose the point $p \in P \subset M$ in the tight pair of pants,
which without loss of generality we may assume to correspond to $0 \in \D^2$ in the lift. Choose an element $f \in \FF_K(M)$ such that $f(p_0) = p$ and denote $\gamma \subset M$ the geodesic homotopic to $f(\gamma_0)$. By our choice of $K_3$, the lift $\widetilde{\gamma}$ of $\gamma$ will intersect at least one of the three arcs $\eta_1, \eta_2$ or $\eta_3$ at an angle $\angle ( \widetilde{\gamma}, \eta_i)_q \geq \pi/4$. 
In other words, $\gamma$ intersects one of the three boundary geodesics $\gamma_1, \gamma_2$ or $\gamma_3$ at an angle at least $\pi/4$.
Further, as $K_2 \leq K_1$, we have that 
\[ K_2(1+\delta_0) \leq K_1 (1+\delta_0) \leq 1 + 3 \delta_0, \] 
as $K_1 (1+\delta_0/2) \leq 1 + \delta_0$, and thus $\gamma$ is $3 \delta_0$-short. This proves the Lemma.
\end{proof}

\subsection{The three-circle Lemma}\label{subsec_three_circle}

A {\em triangle} $T \subset M$ is a subsurface of $M$ such that its boundary consists of a simple closed curve comprised of three geodesic arcs. The triangle $T$ is said to be {\em trivial} if the simple closed curve $\partial T$ is homotopically trivial and {\em non-trivial} otherwise. 
\index{triangle!trivial}
\index{triangle!non-trivial}

\begin{defn}[Three-circle configuration]\label{defn_three_circle}\index{three-circle configuration}
A {\em three-circle configuration} is a union of three $6\delta_0$-short geodesics $\gamma_1, \gamma_2, \gamma_3$, such that
each pair of geodesics intersect in exactly two points and the connected components $M \setminus \bigcup_{j=1}^3 \gamma_j$
consists of exactly $8$ triangles, see Figure~\ref{fig_three_circle}.
\end{defn}

\begin{figure}[h]
\begin{center}
\psfrag{C_1}{$\gamma_2$}
\psfrag{C_2}{$\gamma_1$}
\psfrag{C_3}{$\gamma_3$}
\psfrag{p_1}{$p_1$}
\psfrag{p_2}{$p_2$}
\psfrag{p_3}{$p_3$}
\psfrag{T_1}{$T_1$}
\psfrag{T_2}{$T_2$}
\psfrag{T_3}{$T_3$}
\psfrag{T_4}{$T_4$}
\psfrag{T_6}{$T_6$}
\psfrag{T_5}{$T_5$}
\psfrag{T_7}{$T_7$}
\psfrag{T_8}{$T_8$}
\includegraphics[scale=0.8]{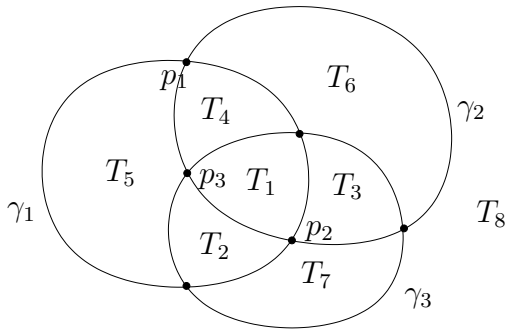}
\caption{Three-circle configuration.}
\label{fig_three_circle}
\end{center}
\end{figure}

\begin{lem}\label{lem_triangle_short_arcs}
Let $\gamma_1, \gamma_2, \gamma_3 \subset M$ comprise a three-circle configuration and let $\{T_j\}_{j=1}^8$ be the collection of triangles associated to the configuration. If a triangle $T_j$ for some $1 \leq j \leq 8$ is trivial, then the length of any of the three geodesic arcs that comprise $\partial T_j$ is at most $\lambda(M)/7$.
\end{lem}

\begin{proof}
In what follows, we write the juxtaposition of arcs to denote the closed curve comprised by concatenating the arcs in counterclockwise direction. Suppose that a triangle $T:= T_j$ for some $1 \leq j \leq 8$ is trivial and let the labeling be as given in Figure~\ref{fig_three_circle_2}, where $\partial T = x_2 z_2 y_2$. The geodesics $\gamma_1, \gamma_2, \gamma_3$ are labeled $\gamma_s, \gamma_t, \gamma_k$ with $1,2,3$ some permutation
of the letters $s,t,k$, where arcs $x_i \subset \gamma_s, y_i \subset \gamma_t$ and $z_i \subset \gamma_k$ with $1 \leq i \leq 3$.

\begin{figure}[h]
\begin{center}
\psfrag{x_1}{$x_1$}
\psfrag{x_2}{$x_2$}
\psfrag{x_3}{$x_3$}
\psfrag{y_1}{$y_1$}
\psfrag{y_2}{$y_2$}
\psfrag{y_3}{$y_3$}
\psfrag{z_1}{$z_1$}
\psfrag{z_2}{$z_2$}
\psfrag{z_3}{$z_3$}
\psfrag{C_1}{$\gamma_s$}
\psfrag{C_2}{$\gamma_t$}
\psfrag{C_3}{$\gamma_k$}
\psfrag{T}{$T$}
\includegraphics[scale=0.8]{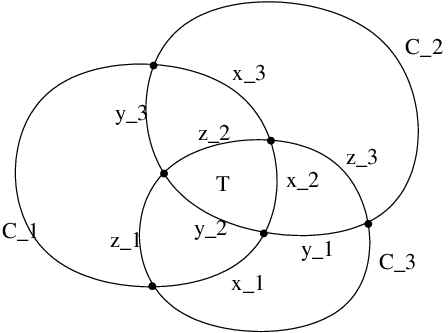}
\caption{Proof of Lemma~\ref{lem_triangle_short_arcs}.}
\label{fig_three_circle_2}
\end{center}
\end{figure}

Define the simple closed curves 
\begin{equation}\label{eq_pf_thma_curves_a}
\alpha_1 = x_2 x_3 y_3 y_2,~ \alpha_2 = x_1 x_2 z_2 z_1~\textup{and}~\alpha_3 = y_2 y_1 z_3 z_2.
\end{equation}
By Lemma~\ref{lem_curves_properties} (iii), the simple closed curves $\alpha_i$, with $1 \leq i \leq 3$, are non-trivial. 
Further, as each of the three geodesics $\gamma_s, \gamma_t$ and $\gamma_k$ is $6 \delta_0$-short (by definition~\ref{defn_three_circle}), 
by the two-circle Lemma we thus have that 
\begin{equation}\label{eq_pf_thma_2}
\lambda(M) \leq | \alpha_i | \leq (1+ 18 \delta_0) \lambda(M),
\end{equation}
for $1 \leq i \leq 3$. Next, consider the simple closed curves
\begin{equation}\label{eq_pf_thma_curves_b}
\beta_1 = x_3 y_3 z_2,~ \beta_2 = x_1 y_2 z_1~\textup{and}~ \beta_3 = y_1 z_3 x_2. 
\end{equation}
As the curves $\alpha_i$ are non-trivial, and the triangle $T$ is trivial, the curve $\beta_i$ is homotopic to $\alpha_i$, for $1 \leq i \leq 3$, 
and therefore non-trivial. Moreover, as $T$ is trivial, by the triangle-inequality applied to $\partial T$, we have that
\begin{equation}\label{eq_triangle_ineq}
|x_2| \leq |y_2| + |z_2|, ~|y_2| \leq |x_2| + |z_2|~\textup{and}~|z_2| \leq |x_2| + |y_2|.
\end{equation}
Combining~\eqref{eq_pf_thma_2} with~\eqref{eq_triangle_ineq}, it is readily verified that 
\begin{equation}
\lambda(M) \leq | \beta_j | \leq (1+ 18 \delta_0) \lambda(M),
\end{equation}
for $1 \leq j \leq 3$. Summing up the lengths of the arcs that constitute the closed curves $\beta_1, \beta_2$ and $\beta_3$, 
cf. \eqref{eq_pf_thma_curves_b}, and reordering the terms, it follows that
\begin{equation}\label{eq_sum_3b}
3 \lambda(M) \leq |x_1| + |x_2| + |x_3| + |y_1| + |y_2| + |y_3| + |z_1| + |z_2| + |z_3| \leq 3(1+ 18 \delta_0) \lambda(M).
\end{equation}
The length of $\alpha_1$ can by~\eqref{eq_pf_thma_curves_a} be expressed as the sum of the lengths of its constituent arcs and estimated by
\begin{equation}\label{eq_sum_alpha_1}
\lambda(M) \leq |x_2| + |x_3| + |y_2| + |y_3| \leq (1+ 18 \delta_0) \lambda(M).
\end{equation}
Subtracting~\eqref{eq_sum_alpha_1} from~\eqref{eq_sum_3b}, we obtain the estimate
\begin{equation}\label{eq_subtract_1}
2\lambda(M) - 18 \delta_0 \lambda(M)  \leq |x_1| + |y_1| + |z_1| + |z_2| + |z_3| \leq 2 \lambda(M) + 54 \delta_0 \lambda(M)
\end{equation}
However, adding up the lengths of $\beta_2$ and $\beta_3$, we must have that
\begin{equation}\label{eq_add_beta}
2 \lambda(M) \leq |x_1| + |x_2| + |y_1| + |y_2| + |z_1| + |z_3| \leq 2(1+ 18 \delta_0) \lambda(M)
\end{equation}
Therefore, subtracting~\eqref{eq_add_beta} from~\eqref{eq_subtract_1}, we obtain
\begin{equation}\label{eq_sum_less_first}
-54 \delta_0 \lambda(M) \leq |z_2| - ( |x_2| + |y_2| ) \leq 54 \delta_0 \lambda(M).
\end{equation}
Repeating the same argument for $\alpha_2$ and $\alpha_3$, one obtains
\begin{eqnarray}\label{eq_sum_less_second}
-54 \delta_0 \lambda(M) \leq |y_2| - ( |x_2| + |z_2| ) \leq 54 \delta_0 \lambda(M). \\ \label{eq_sum_less_last}
-54 \delta_0 \lambda(M) \leq |x_2| - ( |y_2| + |z_2| ) \leq 54 \delta_0 \lambda(M).
\end{eqnarray}
It then follows from~\eqref{eq_sum_less_first},~\eqref{eq_sum_less_second} and~\eqref{eq_sum_less_last} that
\begin{equation}
|x_2| \leq 54 \delta_0 \lambda(M), ~ |y_2| \leq 54 \delta_0 \lambda(M) ~\textup{and}~ |z_2| \leq 54 \delta_0 \lambda(M).
\end{equation}
As $\delta_0 = 1/378$, it thus follows that the lengths of the arcs $x_2,y_2$ and $z_2$ have to be at most $\lambda(M)/7$.
\end{proof}

In Lemma~\ref{lem_exist_three_circle} below, we prove the existence of certain three-circle configurations satisfying additional geometrical 
properties. The proof uses the following geometric estimate. 

\begin{lem}\label{lem_geodesics_intersect_1}
There exists a constant $1 < K_3 \leq K_2$, such that if $M$ is $K$-quasi-conformally homogeneous with $1 < K \leq K_3$, then the following holds. 
Let $\gamma_1, \gamma_2 \subset M$ be two $3 \delta_0$-short geodesics intersecting at a point $p \in \gamma_1 \cap \gamma_2$,
such that $\angle ( \gamma_1, \gamma_2 )_p \geq \pi/4$. Let $\gamma_3 \subset M$ be a geodesic and let $q_0 \in \gamma_3$.
Let $f \in \FF_K(M)$ with $f(p) = q_0$ and let $\gamma_1', \gamma_2'$ be the geodesic homotopic to $f(\gamma_1)$ and $f(\gamma_2)$ respectively. 
Then both $\gamma_1'$ and $\gamma_2'$ are $6 \delta_0$-short and at least one of $\gamma_1'$ or $\gamma_2'$ intersects $\gamma_3$ transversely at a point $q \in \gamma_3$ with $d(q,q_0) \leq 1/20$.
\end{lem}

\begin{proof}
First, by choosing $1 < K_3 \leq K_2$, we have that 
\[ K_3( 1+ 3 \delta_0) \leq 1 + 6 \delta_0, \] 
as $K_2 \leq K_1$ and $K_1( 1+ 3 \delta_0) \leq 1+ 6 \delta_0$. Let $\gamma_1, \gamma_2 \subset M$ be two $3 \delta_0$-short geodesics intersecting at a point $p \in \gamma_1 \cap \gamma_2$, such that $\angle ( \gamma_1, \gamma_2 )_p \geq \pi/4$. Let $\gamma_3 \subset M$ be a geodesic and let $q_0 \in \gamma_3$, where $f(p) = q_0$. Then the geodesics $\gamma_1', \gamma_2'$ homotopic to $f(\gamma_1)$ and $f(\gamma_2)$ respectively are $6 \delta_0$-short, as $\gamma_1$ and $\gamma_2$ are $3 \delta_0$-short. By Lemma~\ref{lem_approx_homeo_disk}, $f$ is approximated by a M\"obius transformation on a compact disk. Further, by Lemma~\ref{lem_approx_geodesic}, the geodesic $\gamma_i'$ stays close to $f(\gamma_i)$, for $i=1,2$. Therefore, by choosing $K_3$ small enough, we have that 
\begin{itemize}
\item[\textup{(i)}] $\gamma_1'$ and $\gamma_2'$ intersect at a point $q_1 \in M$ with $d(q_1,q_0) \leq 1/100$, 
where $q_0 \in \gamma_3$, and
\item[\textup{(ii)}] $\theta := \angle ( \gamma_1', \gamma_2')_{q_1} \geq \frac{\pi}{5}$.
\end{itemize} 

\begin{figure}[h]
\begin{center}
\psfrag{p}{$q_1$}
\psfrag{q'}{$q_2$}
\psfrag{q''}{$q_3$}
\psfrag{eta}{$\eta$}
\psfrag{gamma_1}{$\gamma_1'$}
\psfrag{gamma_2}{$\gamma_2'$}
\psfrag{gamma_3}{$\gamma_3$}
\psfrag{theta}{$\theta$}
\psfrag{alpha}{$\theta'$}
\includegraphics[scale=0.9]{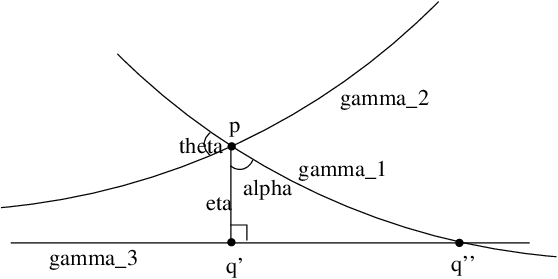}
\caption{Proof of Lemma~\ref{lem_geodesics_intersect_1}.}
\label{fig_lines_intersect}
\end{center}
\end{figure}

Let $\eta \subset M$ be the arc emanating from $q_1$ projecting perpendicularly onto $\gamma_3$ at the point $q_2 \in \gamma_3$. 
As $d(q_1,q_0) \leq 1/100$ and $q_2 \in \gamma_3$, we have that $| \eta | \leq 1/100$. Furthermore, as $\theta \geq \pi/5$, 
at least one of the geodesics $\gamma_i'$ with $i=1,2$ intersects the arc $\eta$ at an angle at most $2\pi/5$. 
Without loss of generality, we may suppose this is the case for $\gamma_1'$, i.e. that 
\[ \theta' := \angle (\eta, \gamma_1')_{q_1} \leq \frac{2\pi}{5}, \] 
see Figure~\ref{fig_lines_intersect}. If we consider the (embedded) geodesic triangle with vertices $q_1, q_2$ and $q_3$, combined with 
$\theta' \leq 2 \pi/5$ and $|\eta| \leq 1/100$, it follows from the hyperbolic sine law that 
\[ d(q_1, q_3) \leq 4/100. \] 
As $d(q_1,q_0) \leq 1/100$, we have that $d(q_0,q_2) \leq 1/100$. Further, we have that $d(q_2, q_3) \leq d(q_1, q_3)$ and thus 
\[ d(q_0, q_3) \leq d(q_0,q_2) + d(q_2, q_3) \leq 1/100 + 4/100 = 1/20. \]
Thus setting $q := q_3$ finishes the proof.
\end{proof}

\begin{lem}[Three-circle Lemma]\label{lem_exist_three_circle}
If $M$ is $K$-quasiconformally homogeneous for $1 < K \leq K_3$, then there exists
a three-circle configuration consisting of simple closed geodesics $\gamma_1, \gamma_2, \gamma_3 \subset M$, such that 
\begin{itemize}
\item[\textup{(i)}] $\gamma_1 \cup \gamma_2$ is a two-circle configuration, and 
\item[\textup{(ii)}] $\gamma_3$ is a $6 \delta_0$-short geodesic intersecting the arc $\eta_3 \subset \gamma_2$ at a point $p_3$ for which  
\[ \left( \frac{1}{4} - \frac{1}{20} \right) \lambda(M) \leq d ( p_j, p_3 ) \leq \left( \frac{1}{4} + \frac{1}{20} \right) \lambda(M), \]
with $j = 1,2$, in the labeling of Figure~\ref{fig_two_circle}, and 
\item[\textup{(iii)}] $\gamma_3$ intersects the interior of the arc $\eta_i$ in exactly one point for every $1 \leq i \leq 4$.
\end{itemize}
\end{lem}

\begin{proof}
As $K_3 \leq K_1$, by Lemma~\ref{lem_two_circles}, there exists a two-circle configuration, comprised of two $\delta_0$-short geodesics 
$\gamma_1, \gamma_2 \subset M$. Label the configuration according to Figure~\ref{fig_two_circle}. Mark a point 
$q \in \eta_3 \subset \gamma_2$ such that $d(q,p_1) = d(q,p_2)$. As $K_3 \leq K_2$, by Lemma~\ref{lem_large_angle}, there exist two
$3 \delta_0$-short geodesics $\gamma_3, \gamma_4$ intersecting at a point $p \in M$, such that 
\[ \angle ( \gamma_3, \gamma_4 )_p \geq \pi/4. \]
Applying Lemma~\ref{lem_geodesics_intersect_1} to $\gamma_3$ and $\gamma_4$ and the target point $q \in \eta_3 \subset \gamma_2$, there exists a $6 \delta_0$-short geodesic $\gamma'$ intersecting $\gamma_2$ transversely at a point $q' \in \gamma_2$ such that $d(q,q') \leq 1/20$. 
Therefore, setting $p_3 := q'$ and $\gamma_3 := \gamma'$, the conditions (i) and (ii) of Lemma~\ref{lem_exist_three_circle} are satisfied. 

We are left with showing that condition (iii) of Lemma~\ref{lem_exist_three_circle} is satisfied. That is, we need to show that $\gamma_3$ intersects 
the interior of the arc $\eta_i$ in exactly one point for every $1 \leq i \leq 4$. To this end, we first show that $\gamma_3$ can not intersect
the arc $\eta_3$ (including the boundary points $p_1$ and $p_2$) more than once. To show this is indeed impossible, suppose that $\gamma_3$ intersects the arc $\eta_3$ in a point $p' \subset \eta_3$ other than $p_3$. Let $\alpha_1, \alpha_2 \subset \eta_3$ be the connected components of $\eta_3 \setminus\{p_3\}$. We may assume that $p' \subset \alpha_1$, the case when $p' \in \alpha_2$ is similar. Therefore, if we let $\alpha' \subset \alpha_1$ the subarc with endpoints $p_3$ and $p'$, where we include the case that $p' = p_1$, then it follows that 
\[ |\alpha'| \leq \left( \frac{1}{4} + \frac{1}{20} \right) \lambda(M). \] 
The two points $p_3$ and $p'$ cut $\gamma_3$ into two component arcs $\beta_1$ and $\beta_2$, one component of which is of 
length at most $(1+ 6 \delta_0)\lambda(M)/2$; without loss of generality, we may suppose this is the case for $\beta_1$. Then the closed 
curve $\alpha' \cup \beta_1$ is homotopically nontrivial and 
\[ | \alpha' \cup \beta_1 | \leq \left( \frac{1}{4} + \frac{1}{20} \right) \lambda(M) + \frac{(1+ 6 \delta_0)\lambda(M)}{2} =  
\left( \frac{3}{4} + \frac{1}{20} + 3 \delta_0 \right) \lambda(M) < \lambda(M), \]
as $\delta_0 = 1/378$, which is a contradiction. Therefore, $\gamma_3$ intersects $\gamma_2$ at the point $p_3$, but does not intersects 
the arc $\eta_3$ in any point other than $p_3$, and $\gamma_3$ does not pass through $p_1$ or $p_2$.

As $\gamma_3$ intersects $\gamma_1$, by Lemma~\ref{lem_curves_properties} (ii), there has to exist at least one more intersection point of $\gamma_3$ 
with $\gamma_2$. By the above argument, all other intersection points are contained in the interior of the arc $\eta_4$. By the two-circle Lemma, applied to $\delta = 6 \delta_0$, $\gamma_3$ intersects $\gamma_2$ only twice, and therefore $\gamma_3$ intersects the interior of $\eta_4$ exactly
once. Similarly, as $\gamma_3$ intersects the arc $\eta_3$ exactly once, $\gamma_3$ has to intersect the interior 
of the arc $\eta_1$ and $\eta_2$ at least once. Again by the two-circle Lemma, applied to $\delta = 6 \delta_0$, as the total number of intersection
points of $\gamma_3$ with $\gamma_1$ is exactly two, $\gamma_3$ has to intersect the interior of the arc $\eta_1$ and $\eta_2$ exactly once. Thus condition (iii) is indeed satisfied and this proves the Lemma.
\end{proof}

\subsection{Proof of Theorem A}\label{subsec_pf_thm_4b}

The endgame of the proof of Theorem A. is a combinatorial argument layered on the three-circle configuration of the Three-Circle Lemma.
Thus, let $1 < K \leq K_3$ with $K$ the quasiconformal homogeneity constant of $M$, with $K_3$ the constant which we obtained in the Three-Circle Lemma in order to ensure the existence of a three-circle configuration. 

\begin{lem}\label{lem_triangle_non_trivial}
Let $\gamma_1, \gamma_2, \gamma_3$ be the three-circle configuration of Lemma~\ref{lem_exist_three_circle}. 
Then the triangle $T_j$ is non-trivial, for every $1 \leq j \leq 8$. 
\end{lem}

\begin{figure}[h]
\begin{center}
\psfrag{C_1}{$\gamma_2$}
\psfrag{C_2}{$\gamma_1$}
\psfrag{C_3}{$\gamma_3$}
\psfrag{T_i}{$T_i$}
\psfrag{T_j}{$T_j$}
\psfrag{x_1}{$y_1$}
\psfrag{x_2}{$y_2$}
\psfrag{x_3}{$y_3$}
\psfrag{x_4}{$y_4$}
\psfrag{p_1}{$p_1$}
\psfrag{p_2}{$p_2$}
\psfrag{p_3}{$p_3$}
\psfrag{p_4}{$p_4$}
\psfrag{p_5}{$p_5$}
\psfrag{p_6}{$p_6$}
\includegraphics[scale=0.8]{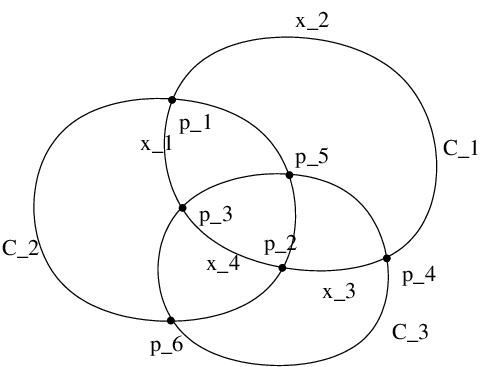}
\caption{Proof of Lemma~\ref{lem_triangle_non_trivial}.}
\label{fig_three_circle_3}
\end{center}
\end{figure}

\begin{proof}
By Lemma~\ref{lem_exist_three_circle}, the $6 \delta_0$-short geodesic $\gamma_3$ intersects $\gamma_2$ at a point 
$p_3 \in \eta_3 \subset \gamma_1$ and 
\begin{equation}\label{eq_lem_triangle_1}
\left( \frac{1}{4} - \frac{1}{20} \right) \lambda(M) \leq d ( p_k , p_3 ) \leq \left( \frac{1}{4} + \frac{1}{20} \right) \lambda(M),
\end{equation}
with $k = 1,2$, as given in Figure~\ref{fig_three_circle_3}. Let $y_i$ with $1 \leq i \leq 4$ be the connected components of 
$\gamma_2 \setminus \{ p_1, p_2, p_3, p_4\}$. It suffices to show that 
\begin{equation}\label{eq_arcs_long_edge}
|y_i| > \frac{\lambda(M)}{6},~ \textup{for}~1 \leq i \leq 4.
\end{equation}
To prove sufficiency, note that every triangle $\partial T_j$, with $1 \leq j \leq 8$, contains exactly one edge $y_i$ for some $1 \leq i \leq 4$.
Now, if~\eqref{eq_arcs_long_edge} holds, then by Lemma~\ref{lem_triangle_short_arcs}, $\partial T_j$ has to be non-trivial, as otherwise
all three edges of $\partial T_j$ have to be of length less than $\lambda/7 < \lambda /6$.

To prove~\eqref{eq_arcs_long_edge}, first note that, by~\eqref{eq_lem_triangle_1}, the arcs $y_1$ and $y_4$ satisfy 
requirement~\eqref{eq_arcs_long_edge}. Therefore, we are left with proving the estimate for $y_2$ and $y_3$. As both $\gamma_1$ and $\gamma_2$ are $\delta_0$-short, 
and $\gamma_3$ is $6 \delta_0$-short, the two-circle Lemma applied to the pairs of geodesics $\gamma_1, \gamma_2$ and $\gamma_2, \gamma_3$ gives respectively
\begin{equation}\label{eq_lem_triangle_2}
\frac{\lambda(M)}{2} - \frac{\delta_0 \lambda(M)}{2} \leq |y_2| + |y_3| \leq \frac{\lambda(M)}{2} + \frac{3\delta_0 \lambda(M)}{2},
\end{equation}
and 
\begin{equation}\label{eq_lem_triangle_3}
\frac{\lambda(M)}{2} - \frac{6\delta_0 \lambda(M)}{2} \leq |y_3| + |y_4| \leq \frac{\lambda(M)}{2} + \frac{18\delta_0 \lambda(M)}{2}.
\end{equation}
Combining~\eqref{eq_lem_triangle_2} and~\eqref{eq_lem_triangle_3}, it follows that
\begin{equation}\label{eq_lem_triangle_4}
-\frac{19 \delta_0 \lambda(M)}{2} \leq |y_2| - |y_4| \leq \frac{19 \delta_0 \lambda(M)}{2}.
\end{equation}
As $|y_4| = d(p_3, p_2)$, combining~\eqref{eq_lem_triangle_4} with~\eqref{eq_lem_triangle_1}, one obtains
\begin{equation}\label{eq_lem_triangle_5}
|y_2| \geq \lambda(M) \left( \frac{1}{4} - \frac{1}{20} - \frac{19 \delta_0}{2} \right) > \frac{\lambda(M)}{6},
\end{equation}
as $\delta_0 = 1/378$. By symmetry, the same estimate holds for the arc $|y_3|$. This concludes the proof.
\end{proof}

Let us now conclude the proof.

\begin{proof}[Proof of Theorem A]
Let $M$ be $K$-quasiconformally homogeneous with $1 < K \leq K_3$ and let $\gamma_1, \gamma_2, \gamma_3$ be the three-circle configuration of 
Lemma~\ref{lem_exist_three_circle}. By Lemma~\ref{lem_triangle_non_trivial}, all triangles $T_j$ with $1 \leq j \leq 8$ have to be non-trivial. Therefore, the length of $\partial T_j$ has to be at least $\lambda(M)$ for every $1 \leq j \leq 8$. Adding up the lengths of all $\partial T_j$, $1 \leq j \leq 8$, means we count every boundary arc of a triangle $T_j$ with multiplicity two and thus 
\begin{equation}\label{eq_final_1}
2 \sum_{i=1}^3 | \gamma_i | = \sum_{j=1}^8 | \partial T_j | \geq 8\lambda(M).
\end{equation}
However, as the geodesics $\gamma_i$, with $1 \leq i \leq 3$, are (at most) $6 \delta_0$-short by construction, the total length of these three geodesics counted with multiplicity two is bounded by
\begin{equation}\label{eq_final_2}
2 \sum_{i=1}^3 | \gamma_i| \leq 2 \cdot 3(1+ 6 \delta_0)\lambda(M) = 6(1+ 6 \delta_0) \lambda(M) < 7 \lambda(M), 
\end{equation}
as $36 \delta_0 < 1$. The contradictory claims~\eqref{eq_final_1} and~\eqref{eq_final_2} finish the proof.
\end{proof}

\section{Concluding remarks}

It would be interesting to obtain an explicit value for the universal constant $\KK > 1$, 
whose existence was shown in the proof of Theorem A. In~\cite{gehring}, examples are given 
of $K$-quasiconformally homogeneous genus zero surfaces, for which 
explicit estimates are computed of the quasiconformal homogeneity constant $K$ of $M$, 
which give the following upper bound on the universal constant $\KK$,
\begin{equation}
1 < \KK \leq \left( \frac{e}{s} \right)^4,
\end{equation}
where $s \approx 0.483$. The geometrical estimates in our proof can in many cases be improved to give 
explicit estimates of the quantities involved, except for the preliminary Lemma 2.2, 
whose known proof is based on a normal family argument (see~\cite{bonfert_1}). 
Using these estimates, can one find a sharper estimate for $\KK$?

\end{document}